\documentclass[12pt]{amsart}%
\usepackage{amsfonts}
\usepackage{txfonts}
\usepackage{mathtools}
\usepackage{amsmath}
\usepackage{amssymb}
\usepackage{amsthm}
\usepackage{newlfont}
\usepackage{galois}
\usepackage{graphicx}%
\providecommand{\U}[1]{\protect\rule{.1in}{.1in}}
\newtheorem{remark}{Remark}
\newtheorem{thm}{Theorem}[section]

\newtheorem{cor}[thm]{Corollary}
\newtheorem{lem}[thm]{Lemma}
\newtheorem{prop}[thm]{Proposition}

\theoremstyle{definition}

\theoremstyle{remark}

\numberwithin{equation}{section}

\newcommand{\R}{\mathbb R}

\newcommand{\ed}{\end {document}}

\begin{document}
\title{Gradient Estimates For $\Delta u + au^{p+1}=0$ And Liouville Theorems}
\author{Bo Peng}
\address{School of
Mathematical Sciences, UCAS, Beijing 100190, China; Institute of Mathematics
The Academy of Mathematics and systems of sciences,Chinese academy of sciences}
\email{pengbo17@mails.ucas.ac.cn}

\author{Youde Wang}
\address{1. College of Mathematics and Information Sciences, Guangzhou University;\\
 2. Hua Loo-Keng Key Laboratory of Mathematics, Institute of Mathematics, Academy of Mathematics and Systems Science, Chinese Academy of Sciences, Beijing 100190, China; 3. School of Mathematical Sciences, University of Chinese Academy of Sciences, Beijing 100049, China.}
\email{wyd@math.ac.cn}

\author{Guodong Wei}
\address{
School of Mathematics (Zhuhai), Sun Yat-sen University, Zhuhai, Guangdong 519082, P. R. China}
\email{weigd3@mail.sysu.edu.cn}
\thanks{ }

\begin{abstract}
In this short note, we use a unified method to consider the gradient estimates of the positive solution to the following nonlinear elliptic equation $\Delta u + au^{p+1}=0$ defined on a complete noncompact Riemannian manifold $(M, g)$ where $a > 0$ and $ p <\frac{4}{n}$ or $a < 0$ and $p >0$ are two constants. For the case $a>0$, this improves considerably the previous known results except for the cases $\dim(M)=4$ and supplements the results for the case $\dim(M)\leq 2$. For the case $a<0$ and $p>0$, we also improve considerably the previous related results. When the Ricci curvature of $(M,g)$ is nonnegative, we also obtain a Liouville-type theorem for the above equation.

\end{abstract}
\maketitle

\section{Introduction}
In the past decades, one pay attention to studying the following elliptic equation defined on a Riemannian manifold $(M, g)$ of dimension $n$
\begin{equation}\label{equ*}
\Delta u(x) +a(x)u(x)^{s} = 0,
\end{equation}
where $s \in \mathbb{R}$. This equation appears in physics. For $a(x) \equiv a < 0$ and $s < 0$, the equation \eqref{equ*} defined on a bounded smooth domain in $\mathbb{R}^{n}$ is called thin film equation, which describes a steady state of the thin film (see \cite{GW}). In addition, when $a(x) \equiv a$ is a constant, it is linked with the theory of stellar structure in astrophysics for $n=3$ and Yang-Mills' problem for $n=4$ and $s=\frac{n+2}{n-3}$ in physics (see \cite{BJ,
MBL}).

On the other hand, as $s=(n+2)/(n-2)$ this equation appears in differential geometry. Let $(M, g)$ be a Riemannian manifold of dimension $n$ ($n \geq 3$), and $K(x)$ be a given function on $M$. If one can find a new metric $g_1$ on $M$ such that $K$ is the scalar curvature of $g_1$ and $g_1$ is conformal to $g$ (i.e., $g_1= u^{4/(n-2}g$ for some function $u > 0$ on $M$), then it is equivalent to there exists a positive solution of the following prescribed scalar curvature equation
$$4(n- 1)/(n-2)\Delta_g u - ku + Ku^{(n+2)/(n-2)},$$
where $\Delta_g$ is the Laplace-Beltrami operator on $M$ with respect to $g$ metric and $k$ is the scalar curvature of $(M, g)$. In the special case where $M\equiv\mathbb{R}^n$ and $g$ is the usual metric, we have $k\equiv 0$ and the above equation reduces to
\begin{equation}\label{con}
\Delta u + K(x)u^{(n+2)/(n-2)}= 0
\end{equation}
in $\mathbb{R}^n$, after an appropriate scaling (see \cite{ding}). Caffarelli, Gidas and Spruck in \cite{CGS} studied non-negative smooth solutions of the conformally invariant equation (\ref{con}) with $K(x)\equiv 1$ in a punctured ball, $B_1(0)\setminus\{0\}\subset\mathbb{R}^n$, $n\geq 3$, with an isolated singularity at the origin. For more details we refer to \cite{ding, dn1, MW, KW, PQS}.

Many mathematicians made contributions to the equation (\ref{equ*}). For example, in the case $a(x) \equiv 0$, \eqref{equ*} is the Laplace equation and the corresponding gradient estimate of \eqref{equ*} has ever been established by Yau in the very famous paper \cite{yau} (for its generalized version, see the remarkable work \cite{C-Y} due to Cheng-Yau).
When $s = 1$, P. Li and S.-T. Yau \cite{LY} proved some results on the gradient estimate of positive solutions under the condition that $|\nabla a(x)| = o(r(x))$ as $r(x)\rightarrow +\infty$ where $r(x)$ is the geodesic distance between $x$ and some fixed point $P$ in $M$.

For $a(x)$ is a real function, the equation \eqref{equ*} is also studied by Gidas and Spruck \cite{BJ} with $1 \leq s < \frac{n+2}{n-2}$ and $n >2$. Under certain conditions on $a(x)$, for instance $a(x) \geq 0$ for $r(x)$ large satisfies that $|\nabla\log a(x)| < C/r(x)$ and if $n\geq 4$, $a(x) > C(r(x))^\sigma$ with $\sigma > -2/(n - 3)$ where $r(x)$ is the distance of $x$ from some fixed point, they proved that any nonnegative solution to the equation is identically zero when the Ricci tensor of manifold is nonnegative.
In particular, they showed that, if $M$ is complete Riemannian manifold with non-negative Ricci tensor and $n>2$, then every non-negative $C^2$ solution of $\Delta u + u^\alpha = 0$, $1\leq\alpha < (n+2)/(n-2)$, on M is identically zero.

Later, Li \cite{J-L} proved the Gidas-Spruck's some results under some weaker restrictions of $a(x)$ for the case $1 < s <\frac{n}{n-2}$ and $n >3$. In particular, Li obtained a gradient estimate for positive solution of \eqref{equ*} when $a(x) \equiv a$ is a positive constant, $1<s<\frac{n}{n-2}$ and $n>3$. Yang \cite{YYY} obtained the corresponding gradient estimates of positive solutions to \eqref{equ*} for $a(x) \equiv a \neq 0$ and $s < 0$, and showed that if $s<0$ and $a(x)$ is a positive constant, \eqref{equ*} does not admit any positive solution on a complete manifold with the nonnegative Ricci tensor (also see \cite{N, W}).

Recently, Ma-Huang-Luo \cite{MHL} studied the gradient estimate for a positive solution when $a(x) \equiv a $ is a positive constant and $\frac{n+2}{2(n-1)}<s<\frac{2n^{2}+9n+6}{2n(n+2)}$ for $n \geq 3$, or $a(x)$ is a positive constant and $s\leq 1$. On the other hand, they also obtained the gradient estimate for a positive solution with $a(x) \equiv a $ is a negative constant and $s>0$. Therefore, it is natural to try to achieve gradient estimates for positive solutions to the nonlinear elliptic equation \eqref{equ*} with $a>0$, other $s$ and dimensions $n$.

In this paper, we focus on studying gradient estimates for positive solutions to the following nonlinear elliptic equation defined on an $n$-dimensional complete noncompact Riemannian manifold $(M,g)$
\begin{equation}\label{equ}
\Delta u(x)+ au(x)^{p+1}=0
\end{equation}
where $a>0$ and $p< \frac{4}{n}$ or $a < 0$ and $p > 0$ are two constant numbers. We try to employ a unified method to obtain a gradient bound of a positive solution to \eqref{equ} which does not depend on such quantities as the bounds of the solution and the Laplacian of the distance function. It is worthy to point out that we do not need to restrict the dimension of the domain manifold. By the authors' knowledge, it seems that there is no results on the equation in the case $\dim(M)=2$. Now we are in the position to state the main results of this paper.

\begin{thm}\label{main}
(Local gradient estimate) Let $\left( M,g\right)$ be an $n$-dimensional complete noncompact Riemannian manifold. Suppose there exists a nonnegative constant $K:=K(2R)$ such that the Ricci Curvature of $M$ is bounded below by $-K$, i.e., $Ric(g) \geq -Kg$ in the geodesic ball $\mathbb{B}_{2R}(O)\subset M$ where $O$ is a fixed point on $M$. If $u$ is a smooth positive solution $u(x)$ to equation \eqref{equ} then, on $\mathbb{B}_{R}(O)$,

\begin{itemize}
\item [(1)] {\bf in the case $a > 0$ and $p < \frac{4}{n},$}
\begin{itemize}
\item [(1.1)] there holds true in the case $p < \frac{2}{n}$
\begin{equation*}
\frac{\left\vert \nabla u \right\vert^{2}}{u^{2}}+au^{p} \leq  \frac{2n}{2-n\max\{0,p\}}\left(\frac{\left((n-1)(1+\sqrt{K}R)+2\right)C_{1}^{2}+C_{2}}{R^{2}}+2K+\frac{2nC_{1}^{2}}{(2+n\max\{0,p\})R^{2}}\right);
\end{equation*}
\item[(1.2)] there holds true in the case $\frac{2}{n} \leq p <\frac{4}{n}$
\begin{equation*}
\frac{\left\vert \nabla u \right\vert^{2}}{u^{2}} \leq \frac{4n}{4-(np-2)^{2}}\left(2K+\frac{\left((n-1)(1+\sqrt{K}R)+2\right)C_{1}^{2}+C_{2}}{R^{2}}+\frac{4n}{4-(np-2)^{2}}\frac{C_{1}^{2}}{R^{2}}\right).
\end{equation*}
\end{itemize}
\item[(2)] {\bf in the case $a < 0$ and $p > 0,$}
\begin{itemize}
\item [(2.1)] there holds true in the case $p \geq \frac{2}{n}$
\begin{equation*}
\frac{\left\vert \nabla u \right\vert^{2}}{u^{2}} \leq n\left(\frac{\left((n-1)(1+\sqrt{K}R)+2\right)C_{1}^{2}+C_{2}}{R^{2}}+2K+\frac{nC_{1}^{2}}{R^{2}}\right);
\end{equation*}
\item[(2.2)] there holds true in the case $0 < p <\frac{2}{n}$
\begin{equation*}
\frac{\left\vert \nabla u \right\vert^{2}}{u^{2}} \leq \frac{4n}{4-(2-pn)^{2}}\left(2K+\frac{\left((n-1)(1+\sqrt{K}R)+2\right)C_{1}^{2}+C_{2}}{R^{2}}+\frac{4n}{4-(2-pn)^{2}}\frac{C_{1}^{2}}{R^{2}}\right).
\end{equation*}
\end{itemize}
\end{itemize}
Here $C_{1}$ and $C_{2}$ are absolute constants independent of the geometry of $M$.
\end{thm}

\begin{remark}
In Theorem \ref{main}, for $a > 0$, it is easy to see that, as $n=4$ we have $(n+4)/n=n/(n-2)$, and there holds true as $n > 4$
$$1+\frac{4}{n}> \frac{n}{n-2}.$$
This means that we extend the range of the corresponding power $s=1+p$ for the same problem as in \cite{J-L} except for $n=4$ since one needs to assume that $1<(1+p)<\frac{n}{n-2}$ and $n>3$ to obtain the gradient estimate in \cite{J-L}.

On the other hand, as $n \geq 3$, we have
$$1+\frac{4}{n}>1+\frac{5n+6}{2n(n+2)}.$$
This means that the range of the corresponding power $s=1+p$ in \cite{MHL} is extended. In addition, comparing with the gradient estimate derived in \cite{MHL}, we don't need the assumption that $n\geq 3$. Our conclusions also cover the results in \cite{YYY}.

Up to now, by the best knowledge of the authors there is no the previous known results for the case $\dim(M)=2$.
\end{remark}

\begin{remark}
For the case $a < 0$ in Theorem \ref{main}, the gradient estimates on a positive solution $u$, obtained in this paper, does not involved the bound of $u$. Meanwhile, the gradient estimates obtained in \cite{MHL} concerns the upper bound of $u$, although they obtained the estimates for the case $a<0$ and $\alpha=1+p>0$. On the other hand, in \cite{YYY} Yang considered the gradient estimate for the case $a<0$ and $\alpha=1+p<0$, also their results involved the lower bound of a positive solution $u$.
\end{remark}

One consequence of Theorem \ref{main} is the following Harnack inequality:
\begin{cor}[\bf Harnack inequality]\label{har}
Suppose the same conditions as in Theorem \ref{main} hold. Let $c(n,p,R,K)$ denote the right hand side of the gradient estimates derived in Theorem \ref{main}(note that $c(n,p,R,K)$ has four possibilities according to the ranges of $a$ and $p$). Then
\[
\sup_{\mathbb{B}_{R/2}(O)}u \leq e^{R\sqrt{c(n,p,R,K)}}\inf_{\mathbb{B}_{R/2}(O)}u.
\]
\end{cor}

By letting $R\rightarrow +\infty$ in Theorem \ref{main}, the following Liouville-type result follows immediately:
\begin{cor}\label{liou}
Let $\left( M,g\right)$ be an $n$-dimensional complete noncompact Riemannian manifold with nonnegative Ricci curvature. Then \eqref{equ} with $(1):a > 0$ and $p < \frac{4}{n}$ or $(2): a <0$ and $p > 0$ does not admit any positive solution.
\end{cor}

\begin{remark}
Here we want to give some remarks of the above Liouville-type results. For the case $a < 0$ and $p >0$, we give a new proof of the Liouville-type theorem on \eqref{equ} defined on a complete noncompact Riemannian manifold with nonnegative Ricci curvature, which was proved for $\Delta u\geq c_0u^\alpha$ ($c_0>0$ and $\alpha>1$ are two real constants), defined on a complete noncompact Riemannian manifold with Ricci curvature bounded from below, by using a completely different method in \cite{CKS}.

For $a > 0$ and $p < \frac{4}{n}$. In \cite{YYY}, Yang has proved that Corollary \ref{liou} holds true with $p<-1$. Recently, Ma-Huang-Luo (see \cite[Corollary 1.3]{MHL}) showed that the similar results holds by assuming $p\leq0$. Therefore, one has known that \eqref{equ} with $-\infty<p<4/(n-2)$ does not admits a positive solution by summarizing the conclusions in \cite{BJ},  \cite{YYY} and \cite{MHL}, if the domain manifold $M$ is a complete noncompact Riemannian manifold with nonnegative Ricci curvature and $\dim(M)\geq 3$.
\end{remark}

However, there are the following two remaining problems for the case $a > 0$ and $p \geq \frac{4}{n}$:

1. Can we obtain the similar gradient estimates on positive solutions to \eqref{equ}?

2. One ask naturally whether \eqref{equ} with $3\leq 1+p <+\infty$, defined on a two dimensional complete noncompact Riemannian manifold with nonnegative Ricci curvature, does not admit any positive solution or not?

\section{Preliminaries}
Throughout this section, we will denote by $\left( M,g\right)$ be an $n$-dimensional complete noncompact Riemannian manifold with $Ric(g) \geq -Kg$ in the geodesic ball $\mathbb{B}_{2R}(O)$, where $K=K(2R) $ is a nonnegative constant depending on $R$ and $O$ is a fixed point on $M$. First, we consider the following equation on $M$
\begin{equation}\label{aug}
\Delta u + auf(\log u) =0,
\end{equation}
where $f\in C^{2}(\R,\R)$ is a $C^2$ function on $\R$ and $a \neq 0$. It is easy to see that if we set $f(t)=\mathrm{e}^{pt}$ , then  \eqref{aug} is actually equal to \eqref{equ}.

 \begin{prop}\label{prop} Let $\left( M,g\right)$ be an $n$-dimensional complete noncompact Riemannian manifold satisfy the same assumption as in Theorem \ref{main}. Suppose that $u(x)$ is a smooth positive solution to equation \eqref{aug} on geodesic ball $\mathbb{B}_{R}(O)$ and let
$$ \omega=\log u\quad \text{and}\quad  G=\left\vert \nabla\omega \right\vert^{2}+\beta f(\omega),$$
here $\beta$ is a constant to be determined later. Then we have
\begin{align*}
\Delta G \geq\, & \frac{2}{n}G^{2} +\bigg( (\beta - 2a)f^{'}(\omega)+\beta f^{''}(\omega)-\frac{4}{n}(\beta -a)f(\omega)-2K\bigg)G\\
& +\frac{2}{n}(\beta -a)^{2}f^{2}(\omega)+2K \beta f(\omega)-\beta (\beta -a)f(\omega)f^{'}(\omega)-\beta^{2}f(\omega)f^{''}(\omega)\\
& -2\left<\nabla\omega,\nabla G\right>.
\end{align*}
\end{prop}

\begin{proof}
First,we notice that there hold
\begin{equation}\label{1}
\Delta\omega+G+(a-\beta)f(\omega)=0,
\end{equation}
and
\begin{equation}\label{2}
\left\vert \nabla\omega \right\vert^{2}=G-\beta f(\omega).
\end{equation}
By the Bochner-Weitzenb\"{o}ck$'$s formula and the assumption on the Ricci curvature tensor, we obtain
\begin{equation}\label{3}
\Delta \left\vert \nabla\omega \right\vert^{2} \geq 2\left\vert \nabla^{2}\omega \right\vert^{2}+2\left<\nabla\omega,\nabla(\Delta\omega)\right>-2K\left\vert \nabla\omega \right\vert^{2}.
\end{equation}
In view of \eqref{1}, \eqref{2} and \eqref{3} we take a direct calculation to obtain
\begin{align*}
\Delta G=\,&\Delta \left\vert \nabla\omega \right\vert^{2}+\Delta \left(\beta f(\omega)\right)\\
\geq \,& 2\left\vert \nabla^{2}\omega \right\vert^{2}+2\left<\nabla\omega,\nabla(\Delta\omega)\right>-2K\left\vert \nabla\omega \right\vert^{2}+\Delta (\beta f(\omega))\\
\geq \,& \frac{2}{n}(\Delta\omega)^{2}+2\left<\nabla\omega,\nabla(\Delta\omega)\right>-2K\left\vert \nabla\omega \right\vert^{2}+\beta\big(f^{''}(\omega)\left\vert \nabla\omega \right\vert^{2}+f^{'}\Delta\omega\big).
\end{align*}
Here we have used the relation
\[
\left\vert \nabla^{2}\omega \right\vert^{2} \geq \frac{1}{n}(\Delta\omega)^{2}
\]
which can be easily derived by Cauchy-Schwarz inequality. Hence, it follows
\begin{align}\label{re}
\Delta G \geq \,& \frac{2}{n}\bigg(G-(\beta-a)f(\omega)\bigg)^{2}-2\left<\nabla\omega,\nabla G \right>-2(a-\beta)f^{'}(\omega)\left\vert \nabla\omega \right\vert^{2}\nonumber\\
& - 2K\left\vert \nabla\omega \right\vert^{2}+\beta f^{''}(\omega)\left\vert \nabla\omega \right\vert^{2}+\beta f^{'}\bigg(-G+(\beta-a)f(\omega)\bigg)\nonumber\\
= & \frac{2}{n}\bigg(G-(\beta-a)f(\omega)\bigg)^{2}-2\left<\nabla\omega,\nabla G \right>-\beta f^{'}(\omega)G\nonumber\\
& + \bigg( 2(\beta -a)f^{'}(\omega)+\beta f^{''}(\omega)-2K\bigg)\bigg(G-\beta f(\omega)\bigg)+\beta(\beta-a)f^{'}(\omega)f(\omega)\nonumber\\
= & \frac{2}{n}G^{2}-2\left<\nabla\omega,\nabla G \right>+\bigg( (\beta-2a)f^{'}(\omega)+\beta f^{''}(\omega)-\frac{4}{n}(\beta-a)f(\omega)-2K\bigg)G\nonumber\\
& + \frac{2}{n}(\beta-a)^{2}f^{2}(\omega)+2K\beta f(\omega)-\beta(\beta-a)f^{'}(\omega)f(\omega)-\beta^{2}f^{''}(\omega)f(\omega).
\end{align}
This is just the required inequality. Thus we complete the proof.
\end{proof}

Next, we will turn to construct the cut-off function. Let $\psi(r)\in C^{2}\left([0,\infty),\R_{\geq0}\right)$ be a $C^{2}$  function on $[0,\infty)$ such that $\psi(r)=1$ for $r \leq 1$,$\psi(r) =0$ for $r \geq 2$,and $0 \leq \psi(r) \leq 1$. Furthermore we can arrange that  $\psi(r)$ satisfying the following
\[
0 \geq \psi^{'}(r) \geq -C_{1}\psi^{\frac{1}{2}}(r) \quad \text{and} \quad \psi^{''}(r) \geq -C_{2}
\]
for some absolute constants $C_{1}$ and $C_{2}$. Now, let
\begin{equation}\label{cut}
\phi(x)=\psi\left(\frac{d(x,O)}{R}\right).\end{equation}
It is easy to see that
$$\phi(x)\big|_{\mathbb{B}_{R}(O)}=1\ \ \text{and}\ \ \phi(x)\big|_{M\backslash\mathbb{B}_{2R}(O)}=0.$$
Furthermore, by using Calabi's trick (see \cite{Calabi}), we can assume without loss of generality that the function $\phi$ is smooth in $\mathbb{B}_{2R}(O)$. Then by the Laplacian comparison theorem (see \cite{SY}),  the following lemma holds obviously on $\mathbb{B}_{2R}(O)$,
\begin{lem}\label{lap} For the function $\phi$ defined as above, there hold true the following two inequalities

\begin{itemize}
\item [(i)]
$$\frac{\left\vert \nabla\phi \right\vert^{2}}{\phi} \leq \frac{C_{1}^{2}}{R^{2}}.$$

\item [(ii)]
$$\Delta \phi \geq -\frac{(n-1)(1+\sqrt{K}R)C_{1}^{2}+C_{2}}{R^{2}}.$$
\end{itemize}
\end{lem}

Now let $x_{0} \in \mathbb{B}_{2R}(O)$ such that
$$
Q=\phi G(x_0)=\sup_{\mathbb{B}_{2R}(O)}\phi G.
$$
We can further assume without loss of generality that $Q>0$, since otherwise Theorem  \ref{main} holds trivially with $\beta=\lambda a$ and $f=\mathrm{e}^{pt}$. Note that $x_0\notin \partial \mathbb{B}_{2R}(O)$. Thus,  at $x_0$, we have
\begin{equation}
\nabla(\phi G)(x_{0})=0,\quad \text{and}  \quad \Delta (\phi G)(x_{0}) \leq 0.
\end{equation}
This  implies
\begin{equation}\label{max}
\phi\nabla G =-G \nabla\phi\ \ \ \text{and} \ \ \ \phi\Delta G + G\Delta\phi - 2G\frac{\left\vert \nabla\phi \right\vert^{2}}{\phi} \leq 0.\end{equation}
Combining Lemma \ref{lap} and \eqref{max} and taking a direct computation yield
\begin{equation}\label{4}
AG \geq \phi\Delta G,
\end{equation}
where
\[
A = \frac{\left((n-1)(1+\sqrt{K}R)+2\right)C_{1}^{2}+C_{2}}{R^{2}}.
\]
On the other hand, from \eqref{2} and \eqref{max}, it is easy to see that at $x_0$ there holds true
\begin{align}\label{5}
-\left<\nabla\omega,\nabla G\right>\phi&=G\left<\nabla\omega,\nabla\phi\right>\nonumber\\
 & = -G\left\vert \nabla\phi \right\vert(G-\beta f(\omega))^{\frac{1}{2}},
\end{align}

Now, by substituting \eqref{4} and \eqref{5} into \eqref{re}, we obtain
\begin{eqnarray}\label{6}
AG &\geq & \frac{2}{n}\phi G^{2} +\bigg( (\beta - 2a)f^{'}(\omega)+\beta f^{''}(\omega)-\frac{4}{n}(\beta -a)f(\omega)-2K\bigg)\phi G\nonumber\\
&& +\bigg( \frac{2}{n}(\beta -a)^{2}f^{2}(\omega)+2K \beta f(\omega)-\beta (\beta -a)f(\omega)f^{'}(\omega)-\beta^{2}f(\omega)f^{''}(\omega)\bigg)\phi\nonumber\\
&& -2G\left\vert \nabla\phi \right\vert(G-\beta f(\omega))^{\frac{1}{2}}.
\end{eqnarray}

\section{The Proof of Theorem \ref{main}}
In this section we give the proof of main theorem.
\begin{proof}[\bf Proof of Theorem \ref{main}]
Letting $\beta = \lambda a$ (where $\lambda $ is to be determined) and $f(t)=\mathrm{e}^{pt}$ in \eqref{6}, we can obtain from Proposition \ref{prop} and the above arguments that
\begin{equation}
\begin{split}
AG \geq &\frac{2}{n}\phi G^{2} +\left( \lambda p^{2}+(\lambda -2)p-\frac{4}{n}(\lambda-1)\right)a\mathrm{e}^{p\omega}\phi G\\
& +\left( \left(\frac{2}{n}(\lambda-1)^{2}-\lambda(\lambda-1)p-\lambda^{2}p^{2}\right)a^{2}\mathrm{e}^{2p\omega}+2K\lambda a \mathrm{e}^{p\omega}\right)\phi\\
& - 2K\phi G-2 \left\vert \nabla\phi \right\vert G\left(G-\lambda a \mathrm{e}^{p\omega}\right)^{\frac{1}{2}}.
\end{split}
\end{equation}
Denote
\[
T=\lambda p^{2}+(\lambda -2)p-\frac{4}{n}(\lambda-1) \quad and \quad N=\frac{2}{n}(\lambda-1)^{2}-\lambda(\lambda-1)p-\lambda^{2}p^{2},
\]
we have
\begin{equation}\label{0}
AG \geq \frac{2}{n}\phi G^{2} +Tau^{p}\phi G +\left( Na^{2}u^{2p}+2K\lambda au^{p}\right)\phi-2K\phi G-2 \left\vert \nabla\phi \right\vert G(G-\lambda au^{p})^{\frac{1}{2}}.
\end{equation}
According to the sign of $a$, we need to consider the following two cases:

\

\noindent{\bf Case $1$}: $a > 0$. In this case, we need to handle two cases: (1). $\lambda > 0$ and (2). $\lambda = 0$.\medskip

For the case $\lambda > 0$, by using Young inequality we can see easily that there hold true
\begin{equation}\label{24}
2 \left\vert \nabla\phi \right\vert G(G-\lambda au^{p})^{\frac{1}{2}} \leq C \frac{\left\vert \nabla \phi \right\vert^{2}}{\phi}G +\frac{1}{C}\phi G(G-\lambda a u^{p}),
\end{equation}
here $C$ is a positive constant to be determined later. By substituting \eqref{24} into \eqref{0}, we derive
\begin{align*}
AG \geq \, & \frac{2}{n}\phi G^{2} +Tau^{p}\phi G +\bigg(Na^{2}u^{2p}+2K\lambda au^{p}\bigg)\phi-2K\phi G-C\frac{\left\vert \nabla \phi \right\vert^{2}}{\phi}G -\frac{1}{C}\phi G(G-\lambda a u^{p})\\
= \, & \bigg(\frac{2}{n}-\frac{1}{C}\bigg)\phi G^{2} -2K\phi G-C\frac{\left\vert \nabla \phi \right\vert^{2}}{\phi}G + \bigg(\bigg(T+\frac{\lambda}{C}\bigg)G+ Nau^{p}+2K\lambda\bigg)au^{p}\phi.
\end{align*}
If $N \leq 0$ and $T+\frac{\lambda}{C}+\frac{N}{\lambda} \geq 0$, then, we have
\begin{align*}
AG  & \, \geq \bigg(\frac{2}{n}-\frac{1}{C}\bigg)\phi G^{2} -2K\phi G-C\frac{\left\vert \nabla \phi \right\vert^{2}}{\phi}G + \bigg(\bigg(T+\frac{\lambda}{C}\bigg)G+ Nau^{p}+2K\lambda\bigg)au^{p}\phi\\
& \, \geq \bigg(\frac{2}{n}-\frac{1}{C}\bigg)\phi G^{2} -2K\phi G-C\frac{\left\vert \nabla \phi \right\vert^{2}}{\phi}G + \bigg(\bigg(T+\frac{\lambda}{C}\bigg)G+ N\frac{G}{\lambda}+2K\lambda\bigg)au^{p}\phi\\
& \, = \bigg(\frac{2}{n}-\frac{1}{C}\bigg)\phi G^{2} -2K\phi G-C\frac{\left\vert \nabla \phi \right\vert^{2}}{\phi}G + \bigg(\bigg(T+\frac{\lambda}{C}+\frac{N}{\lambda}\bigg)G+2K\lambda\bigg)au^{p}\phi\\
& \, \geq \bigg(\frac{2}{n}-\frac{1}{C}\bigg)\phi G^{2} -2K\phi G-C\frac{C_{1}^{2}}{R^{2}}G,
\end{align*}
where we have used the fact $au^{p} \leq \frac{G}{\lambda}$.
When $\frac{2}{n}-\frac{1}{C} >0$, multiplying both side of the last inequality by $1/G$, we obtain
\begin{equation}
\phi G \leq \frac{A+2K+C\frac{C_{1}^{2}}{R^{2}}}{\frac{2}{n}-\frac{1}{C}}.
\end{equation}
To maintain the above inequalities hold true, we need to solve the following inequalities
\begin{equation}
\left\{
\begin{aligned}
&T+\frac{\lambda}{C}+\frac{N}{\lambda} \geq 0,\\
&N=\frac{2}{n}(\lambda-1)^{2}-\lambda(\lambda-1)p-\lambda^{2}p^{2} \leq 0,\\
&\lambda > 0,\ \ \text{and}\ \ \ \frac{2}{n}-\frac{1}{C} >0.
\end{aligned}
\right.
\end{equation}
Calculating directly, we have
\begin{equation}
\left\{
\begin{aligned}
& p \leq \frac{2}{n}\cdot\frac{1}{\lambda}-\left(\frac{2}{n}-\frac{1}{C}\right)\lambda,\\
& p \leq \frac{-\sqrt{1+\frac{8}{n}}\left\vert \lambda-1 \right\vert-(\lambda-1)}{2\lambda} \quad \mbox{or} \quad p \geq \frac{\sqrt{1+\frac{8}{n}}\left\vert \lambda-1 \right\vert-(\lambda-1)}{2\lambda},\\
&\lambda > 0,\quad \text{and}\quad \frac{2}{n}-\frac{1}{C} >0.
\end{aligned}
\right.
\end{equation}
Now we choose $\lambda = 1$ and $\frac{1}{C}= \frac{\max\{0,p\}}{2}+\frac{1}{n}$. Then, for $p \in(-\infty,\frac{2}{n})$, there holds
\begin{equation}\label{26}
\sup_{\mathbb{B}_{R}} \left(\frac{\left\vert \nabla u \right\vert^{2}}{u^{2}}+au^{p}\right) \leq \frac{2n}{2-n\max\{0,p\}}\left(A+2K+\frac{2nC_{1}^{2}}{(2+n\max\{0,p\})R^{2}}\right).
\end{equation}

\medskip

For the case $\lambda = 0$, in this case, notice that $0 \leq \phi \leq 1$ we obtain
\begin{eqnarray}\label{21}
AG  &\geq & \frac{2}{n}\phi G^{2} +Tau^{p}\phi G +Na^{2}u^{2p}\phi -2K\phi G-2\left\vert \nabla \phi \right\vert G^{\frac{3}{2}}\nonumber\\
&\geq & \frac{2}{n}\phi G^{2} -2\left\vert \nabla \phi \right\vert G^{\frac{3}{2}}-2K\phi G+\left( T G +N a u^{p}\right) au^{p}\phi\nonumber\\
& = & \left(\sqrt{N}au^{p}+\frac{TG}{2\sqrt{N}}\right)^{2}\phi +\left(\frac{2}{n}-\frac{T^{2}}{4N}\right)\phi G^{2}-2\frac{C_{1}}{R}\phi^{\frac{1}{2}}G^{\frac{3}{2}}-2K G\nonumber\\
& \geq & \left(\frac{2}{n}-\frac{T^{2}}{4N}\right)\phi G^{2}-2\frac{C_{1}}{R}\phi^{\frac{1}{2}}G^{\frac{3}{2}}-2K G.
\end{eqnarray}
For $\frac{2}{n}-\frac{T^{2}}{4N} >0$, it is easy to see that the above inequality implies
\begin{equation}\label{22}
\phi G \leq \frac{8nN}{8N-nT^{2}}\left(2K+A+\frac{8nN}{8N-nT^{2}}\frac{C_{1}^{2}}{R^{2}}\right).
\end{equation}
Therefore, we need to solve the following inequalities
\begin{equation}
\left\{
\begin{aligned}
&\frac{2}{n}-\frac{T^{2}}{4N} >0,\\
&N=\frac{2}{n} > 0.
\end{aligned}
\right.
\end{equation}
By taking a direct calculation, we derive
\begin{equation}\label{r}
0< p < \frac{4}{n}.
\end{equation}
Now for $\frac{2}{n} \leq p < \frac{4}{n}$, then it is easy to see from \eqref{22} and \eqref{r} that
\begin{equation}\label{27}
\sup_{\mathbb{B}_{R}}G \leq \frac{4n}{4-(np-2)^{2}}\left(2K+A+\frac{4n}{4-(np-2)^{2}}\frac{C_{1}^{2}}{R^{2}}\right),
\end{equation}
Hence, by summarizing the previous \eqref{26} and \eqref{27}, we complete the proof of Theorem \ref{main} for case $a > 0$.

\medskip

\noindent{\bf Case $2$}: $a < 0$.

In this case, we choose $\lambda = 0$. It is easy to check that \eqref{21} still holds true. Then for $T \leq 0$ and $N > 0$ we have
\[
AG \geq \frac{2}{n}\phi G^{2} -2\frac{C_{1}}{R}\phi^{\frac{1}{2}}G^{\frac{3}{2}}-2K G.
\]
Multiplying both side of the last inequality by $1/G$, it follows immediately from Young's inequality and the above inequality that
\[
\phi G \leq n\left(2K + A +\frac{nC_{1}^{2}}{R^{2}}\right).
\]
Calculating directly, we obtain
\[
p \geq \frac{2}{n}.
\]
Hence for $p \geq \frac{2}{n}$, we have
\begin{equation}\label{28}
\phi G \leq n\left(2K + A +\frac{nC_{1}^{2}}{R^{2}}\right).
\end{equation}
Similarly for $\frac{2}{n}-\frac{T^{2}}{4N} >0, $ we can see easily from the Young's inequality that
\begin{equation}\label{31}
\phi G \leq \frac{8nN}{8N-nT^{2}}\left(2K+A+\frac{8nN}{8N-nT^{2}}\frac{C_{1}^{2}}{R^{2}}\right).
\end{equation}
Therefore, we need to solve the following inequalities
\begin{equation}
\left\{
\begin{aligned}
&\frac{2}{n}-\frac{T^{2}}{4N} >0,\\
&N > 0.
\end{aligned}
\right.
\end{equation}
By taking a direct calculation, we derive
\begin{equation}\label{r1}
0 < p < \frac{4}{n}
\end{equation}
Thus, for $0 <p < \frac{2}{n}$,  it is easy to see from \eqref{31} and \eqref{r1} that
\begin{equation}\label{32}
\sup_{\mathbb{B}_{R}}G \leq \frac{4n}{4-(2-np)^{2}}\left(2K+A+\frac{4n}{4-(2-np)^{2}}\frac{C_{1}^{2}}{R^{2}}\right),
\end{equation}
Hence, by summarizing the previous \eqref{28} and \eqref{32}, we complete the proof of Theorem \ref{main}.
\end{proof}

\begin{proof}[\bf Proof of Corollary \ref{har}]
It follows from Theorem \ref{main} that on $\mathbb{B}_{R}(O)$,
$$\frac{|\nabla u|}{u}\leq \sqrt{c(n,p.R,K)}.$$
Choose $x,y \in \mathbb{B}_{R/2}(O)$ such that
$$u(x)=\sup_{\mathbb{B}_{R/2}(O)}u(x)\ \ \text{and}\ \ u(y)=\inf_{\mathbb{B}_{R/2}(O)}u(x).$$
Let $\gamma(t), t\in [0,l]$ be a shortest curve with arc length in $(M,g)$ connecting $y$ and $x$ with $\gamma(0)=x,\ \gamma(l)=y$. By the triangle inequality, we can see easily that $\gamma\in \mathbb{B}_{R}(O)$ and $l\leq R$. Then
\begin{align*}
\log u(x) -\log u(y) &=\int_{0}^{l}\frac{\partial \log u\comp \gamma(t)}{\partial t} dt\\&\leq \int_{\gamma} \ \frac{\left\vert \nabla u \right\vert}{u} \\
& \leq \sqrt{c(n,p,R,K)}\cdot R.
\end{align*}
Hence
\[
\sup_{\mathbb{B}_{R/2}(O)}u \leq e^{R\sqrt{C(n,p,R,K)}}\inf_{\mathbb{B}_{R/2}(O)}u.
\]
We finish the proof.
\end{proof}
\medskip

\noindent {\it\textbf{Acknowledgements}}: The authors are supported partially by NSFC grant (No.11731001). The author Y. Wang is supported partially by NSFC grant (No.11971400) and Guangdong Basic and Applied Basic Research Foundation Grant (No. 2020A1515011019). The third author is supported partially by China Postdoctoral Science Foundation Grant (No. 2019M660274).

\bibliographystyle{amsalpha}

\end{document}